\documentclass[11pt]{amsart}
\usepackage{amsmath}
\usepackage{amssymb}
\usepackage{amsthm}
\usepackage{latexsym}

\def\NZQ{\mathbb}               % the font for N,Z,Q,R,C
\def\NN{{\NZQ N}}

\def\F2{{\NZQ F}_2}

\def\opn#1#2{\def#1{\operatorname{#2}}} % to make operators
\opn\chara{char} \opn\length{\ell} \opn\pd{pd} \opn\rk{rk}
\opn\projdim{proj\,dim} \opn\injdim{inj\,dim} \opn\rank{rank}
\opn\depth{depth} \opn\codepth{codepth} \opn\grade{grade}
\opn\height{height} \opn\embdim{emb\,dim} \opn\codim{codim}

\opn\Tr{Tr} \opn\bigrank{big\,rank}
\opn\superheight{superheight}\opn\lcm{lcm}
\opn\trdeg{tr\,deg}%
\opn\reg{reg} \opn\lreg{lreg} \opn\skel{skel}
\opn\Gr{Gr}
\opn\ann{ann}
\opn\sign{sign}
\opn\del{del}

\opn\div{div} \opn\Div{Div} \opn\cl{cl} \opn\Cl{Cl}
\opn\Spec{Spec} \opn\Supp{Supp} \opn\supp{supp} \opn\Sing{Sing}
\opn\Ass{Ass}\opn\fdepth{fdepth}
\opn\Ann{Ann} \opn\Rad{Rad} \opn\Soc{Soc}
\opn\Sym{Sym} \opn\Ker{Ker} \opn\Coker{Coker} \opn\Im{Im}
\opn\Hom{Hom} \opn\Tor{Tor} \opn\Ext{Ext} \opn\End{End}
\opn\Aut{Aut} \opn\id{id} \opn\ini{in} \opn\tr{tr}

\opn\nat{nat}\opn\it{it}
\opn\pff{proof}%   \pf exists already
\opn\Pf{proof} \opn\GL{GL} \opn\SL{SL} \opn\mod{mod} \opn\ord{ord}
\opn\aff{aff} \opn\con{conv} \opn\relint{relint} \opn\st{st}
\opn\lk{lk} \opn\cn{cn} \opn\core{core} \opn\vol{vol}
\opn\link{link} \opn\star{star} \opn\skel{skel} \opn\indeg{indeg}
\opn\Ass{Ass} \opn\Min{Min} \opn\sdepth{sdepth} \opn\depth{depth}
\opn\gr{gr}

\def\pot#1#2{#1[\kern-0.28ex[#2]\kern-0.28ex]}

\opn\dirlim{\underrightarrow{\lim}}
\opn\inivlim{\underleftarrow{\lim}}
\def\Implies{\ifmmode\Longrightarrow \else
     \unskip${}\Longrightarrow{}$\ignorespaces\fi}
\def\implies{\ifmmode\Rightarrow \else
     \unskip${}\Rightarrow{}$\ignorespaces\fi}
\def\iff{\ifmmode\Longleftrightarrow \else
     \unskip${}\Longleftrightarrow{}$\ignorespaces\fi}

\let\:=\colon
\opn\d{d}
\newtheorem{Theorem}{Theorem}[section]
\newtheorem{Lemma}[Theorem]{Lemma}
\newtheorem{Corollary}[Theorem]{Corollary}
\newtheorem{Proposition}[Theorem]{Proposition}

\newtheorem{Example}[Theorem]{Example}

\newtheorem{Remark}[Theorem]{Remark}
\let\epsilon\varepsilon
\let\phi=\varphi
\let\kappa=\varkappa
\def\qed{\ifhmode\textqed\fi
   \ifmmode\ifinner\quad\qedsymbol\else\dispqed\fi\fi}
\def\textqed{\unskip\nobreak\penalty50
    \hskip2em\hbox{}\nobreak\hfil\qedsymbol
    \parfillskip=0pt \finalhyphendemerits=0}
\def\dispqed{\rlap{\qquad\qedsymbol}}

\opn\Gin{Gin}

\def\FF{{\mathcal F}}

\newcommand{\ri}{\mathrm{ri}}
\newcommand{\HP}{\mathrm{HP}}

\opn\inii{in} \opn\inim{inm} \opn\rate{rate}
\opn\H{H}

\numberwithin{equation}{section}

\title{On the reduced Euler characteristic of independence complexes of circulant graphs}
\keywords{}
\usepackage{graphicx}
\usepackage{xcolor}
\usepackage{tikz}
\author{Giancarlo Rinaldo}
\address{Department of Mathematics\\
University of Trento\\
via Sommarive, 14\\
38123 Povo (Trento), Italy
}
\author{Francesco Romeo}
\begin{document}
\begin{abstract}
Let $G$ be the circulant graph $C_n(S)$ with $S\subseteq\{ 1,\ldots,\left \lfloor\frac{n}{2}\right \rfloor\}$. We study the reduced Euler characteristic $\tilde{\chi}$ of the independence complex $\Delta (G)$ for $n=p^k$ with $p$ prime and for $n=2p^k$ with $p$ odd prime, proving that in both cases $\tilde{\chi}$ does not vanish. We also give an example of circulant graph whose independence complex has $\tilde{\chi}$ equals to $0$, giving a negative answer to R. Hoshino.

\smallskip
\noindent \textbf{Keywords.} Circulant graph, Euler Characteristic, Simplicial Complex. 
\end{abstract}

\maketitle

\section*{Introduction}\label{sec:intro}
Let $G$ be a finite simple graph with vertex set $V(G)$ and edge set $E(G)$. A subset $C$ of $V(G)$ is a \emph{clique} of $G$ if  any two different vertices of $C$ are adjacent in $G$. A subset $A$ of $V(G)$ is called an \textit{independent set} of $G$ if no two vertices of $A$ are adjacent in $G$. The \textit{complement graph} of $G$, $\bar{G}$, is the graph with vertex set $V(G)$ and edge set $E(\bar{G})= \{\{u,v\} \mbox{ with } u,v \in V(G) \mid \{u,v\}\notin E(G)\}$.
In particular, a set is independent in $G$ if and only if it is a clique in the complement graph $\bar{G}$.

\noindent We also recall that a circulant graph is defined as follows. Let $S\subseteq \{ 1,2,\ldots,\left \lfloor\frac{n}{2}\right \rfloor\}$. The \textit{circulant graph} $G:=C_n(S)$ is a simple graph with  $V(G)=\mathbb{Z}_n=\{0,\ldots,n-1\}$ and $E(G) := \{ \{i, j\} \mid |j-i|_n \in S \}$ where $|k|_n=\min\{|k|,n-|k|\}$.

\noindent Recently many authors have studied some combinatorial and algebraic properties of circulant graphs (see \cite{Ho}, \cite{BH1}, \cite{BH0}, \cite{MTW}, \cite{EMT}, \cite{Ri}).  In particular, in \cite{Ho}, \cite{BH1}, \cite{BH0}, \cite{EMT}, a formula for the $f$-vector of the independence complex was showed for some nice classes of circulants, e.g. the $d$-th power cycle, $S=\{1,2,\ldots,d\}$, and its complement.  Moreover, Hoshino in \cite[p. 247]{Ho} focused on the Euler characteristic, an invariant that is associated to any simplicial complex (see \cite{BH2}). In particular, he conjectured, by our notation, that any independence complex associated to a non-empty circulant graph has reduced Euler characteristic always different from $0$.

%https://en.wikipedia.org/wiki/Null_graph la parte Edgeless graph.

\noindent We show that for particular $n$, a circulant graph $C_n(S)$ will support the conjecture, independent of the choice on $S$. To this aim, we exploit that each entry of the $f$-vector is a multiple of a divisor of $n$ (see Remark \ref{Div}).

\noindent In Section \ref{mainsection} we prove that the conjecture holds for $n=p^k$ for any prime $p$, and for $n=2p^{k}$ for any odd prime $p$. Moreover we disprove the conjecture providing a counterexample (see Example \ref{counterexample}).

As an application of our results, we focus our attention on two algebraic objects related to the independence complex of circulant graphs. We first consider the \emph{independence polynomial} (see \cite{Ho},\cite{BH0}), that is
\begin{equation}\label{indep}
I(G,x)= \sum\limits_{i=0}^{n} f_{i-1}x^{i},
\end{equation}
where $f_{i-1}$ are the entries of the $f$-vector of the independence complex of a graph $G$. In particular, $-1$ is a root of the independence polynomial if and only if the Euler characteristic of the independence complex vanishes. This happens in Example \ref{counterexample} and does not happen for all the cases studied in Theorems \ref{Th}, \ref{Th2}. 

\noindent The second one arises from commutative algebra (see also \cite{BH2}, \cite{MS}, \cite{Vi0}, \cite{St}). Let $R=K[x_{0},\ldots ,x_{n-1}]$ be the polynomial ring and $I(G)$ the edge ideal related to the graph $G$ (see \cite{Vi}), that is
\begin{equation}\label{edgid}
I(G)=(x_ix_j:\{i,j\}\in E(G)).
\end{equation}
In this case the non-vanishing of the reduced Euler charateristic gives us information about the regularity index of $R/I(G)$, namely the smallest integer such that the Hilbert function on $R/I(G)$ becomes a polynomial function, the so-called Hilbert polynomial (see Section \ref{sec:pre}, Remark \ref{remri}). Also in this case Theorems \ref{Th}, \ref{Th2} and Example \ref{counterexample} are relevant.

\section{Preliminaries}\label{sec:pre}
In this section we recall some concepts and notations on graphs and on simplicial complexes that we will use in the article. 

\noindent Set $V = \{x_1, \ldots, x_n\}$. A \textit{simplicial complex} $\Delta$ on the vertex set $V$ is a collection of subsets of $V$ such that: 1) $\{x_i\} \in \Delta$ for all $x_i \in V$; 2) $F \in \Delta$ and $G\subseteq F$ imply $G \in \Delta$.
An element $F \in \Delta$ is called a \textit{face} of $\Delta$. A maximal face of $\Delta$ with respect to inclusion is called a \textit{facet} of $\Delta$.

\noindent
The dimension of a face $F \in \Delta$ is $\dim F = |F|-1$, and the dimension of $\Delta$ is the maximum of the dimensions of all facets. Let $d-1$ be the dimension of $\Delta$ and let $f_i$ be the number of faces of $\Delta$ of dimension $i$ with the convention that $f_{-1}=1$. Then the $f$-vector of $\Delta$ is the $(d+1)$-tuple $f(\Delta)=(f_{-1},f_0,\ldots,f_{d-1})$. The $h$-vector of $\Delta$ is $h(\Delta)=(h_0,h_1,\ldots,h_d)$ with
\begin{equation} \label{hve}
 h_k=\sum_{i=0}^{k}(-1)^{k-i}\binom{d-i}{k-i} f_{i-1}. 
\end{equation}
The sum
\[
 \widetilde{\chi}(\Delta)=\sum_{i=0}^{d}(-1)^{i-1}f_{i-1}
\]
is called the reduced Euler characteristic of $\Delta$ and $h_d=(-1)^{d-1}\widetilde{\chi}(\Delta)$.

\noindent Given any simplicial complex $\Delta$ on $V$, we can associate a monomial ideal $I_\Delta$ in the polynomial ring $R$ as follows:
\[
 I_\Delta=(\{x_{j_1}x_{j_2}\cdots x_{j_r}: \{x_{j_1},x_{j_2},\ldots,x_{j_r}\}\notin \Delta\}).
\]
$R/I_\Delta$ is called \emph{Stanley-Reisner ring} and its Krull dimension is $d$. If $G$ is a graph, the \textit{independence complex} of $G$ is
\[
\Delta(G)=\{A\subset V(G): A \mbox{ is an independent set of }G\}. 
\]
The independence polynomial is associated to $\Delta (G)$ and by Equation \eqref{indep} it follows
\begin{equation}\label{charind}
\widetilde{\chi}(\Delta (G)) = -I(G,-1)
\end{equation}
We also remark that from the definition of Stanley-Reisner ring and by Equation \eqref{edgid}, it follows $R/I_{\Delta (G)}=R/I(G)$.

The \textit{clique complex} of a graph $G$ is the simplicial complex whose faces are the cliques of $G$.

\begin{Remark}
Let $G=C_{n}(S)$ be a circulant graph on $S \subseteq T:=\{1,2,\ldots,\left \lfloor\frac{n}{2}\right \rfloor\}$. We observe that the complement graph of $G$, namely $\bar{G}$, is a circulant graph on $\bar{S}:=T\setminus S$. Moreover the clique complex of $\bar{G}$ is the independence complex of $G$, $\Delta(G)$.
\end{Remark}

\noindent We also recall some basic facts about the regularity index (see \cite[Chapter 5]{Vi0}). Let $R$ be a standard graded ring and $I$ be a homogeneous ideal. The \textit{Hilbert function} $\H_{R/I} : \mathbb{N} \rightarrow \mathbb{N}$ is defined by 
\[
\H_{R/I} (k) := \dim_K (R/I)_k
\]
where $(R/I)_k$ is the $k$-degree component of the gradation of $R/I$ (see \cite[Section 2.2]{Vi}), while the Hilbert-Poincar\'e series of $R/I$ is
\[
\HP_{R/I} (t) := \sum_{k \in \NN} \H_{R/I}(k) t^k. 
\]
By the Hilbert-Serre theorem, the Hilbert-Poincar\'e series of $R/I$ is a rational function, in particular
\[
\HP_{R/I}(t) = \frac{h(t)}{(1-t)^n}.
\]
for some $h(t) \in \mathbb{Z}(t)$.
There exists a unique polynomial $P_{R/I}$ such that $\H_{R/I}(k) = P_{R/I}(k)$ for all $k\gg 0$. 
The minimum integer $k_0 \in \NN$ such that $\H_{R/I}(k) = P_{R/I}(k)$ for all $k \geq k_0$ is called \textit{regularity index} and we denote it by $\ri(R/I)$.\\
We end this section with the following
\begin{Remark}\label{remri}
 Let $R/I_\Delta$ be a Stanley-Reisner ring. Then 
 \[ 
\ri(R/I_\Delta)=
\left \{
  \begin{tabular}{cl}
  $0$ & if $h_d=0$  \\
  $1$ & if $h_d\neq 0$.  
  \end{tabular}
\right. 
\]
\end{Remark}
Related to the regularity index is the $a$-invariant (see Chapter 5 of \cite{Vi0}), namely the degree of $\HP_{R/I}(t)$ as a rational function, that gives further information about other algebraic invariants.

\section{Reduced Euler characteristic of the independence complex \\ of some circulants}\label{mainsection}
The goal of this section is to study the reduced Euler characteristic, $\widetilde{\chi}$, of the independence complex $\Delta(G)$ of circulant graphs, by proving bounds on the maximum clique number $\omega (\bar{G})$.
In \cite{Ri} the author proves that $\widetilde{\chi}(\Delta(G))\neq 0$ when $n$ is a prime number. We generalize the result  for $n=p^k$ for any prime $p$, and $n=2p^k$ for any odd prime $p$. For the sake of completeness, we give the following Lemma that has been stated in \cite[Lemma 1]{Ri}.

\begin{Lemma}\label{Div}
Let $G$ be a circulant graph on $n$ vertices. Let $f_{i-1}$ be the number of independent sets of cardinality $i$, and $f_{i-1,0}$ the number of them containing the vertex 0, then the following property holds
\[
i\cdot f_{i-1}=n \cdot f_{i-1,0} \\ \mbox{ with } \\ 0\leq i\leq d.
\]
\end{Lemma}
\begin{proof}
Let us call $\FF_{i-1}\subset \Delta$ the set of faces of dimension $i-1$, that is 
\[
\FF_{i-1}=\{F_1,\ldots,F_{f_{i-1}}\}. 
\]
Let $f_{i-1,j}$ be number of faces in $\FF_{i-1}$  containing a given vertex $j=0,\ldots, n-1$. Since $G$ is circulant
$f_{i-1,j}=f_{i-1,0} \mbox{ for all }j\in \{0,\ldots,n-1\}$.
Let $A\in \F2^{f_{i-1} \times n}$, $A=(a_{jk})$, be the \emph{incidence matrix}, whose 
\[
a_{jk} = \begin{cases} 1 &\mbox{if the vertex } k-1 \in F_j \\ 0 &\mbox{otherwise.} \end{cases}
\]
We observe that each row has exactly $i$ $1$-entries. Hence summing the entries of the matrix we have $i f_{i-1}$. Moreover each column has exactly $f_{i-1,j}$ non-zero entries. That is $i\cdot f_{i-1} = n\cdot f_{i-1,0}$.
\end{proof}
\noindent A useful bound on the maximum clique number for non-complete circulant graphs is given by the following
\begin{Lemma}\label{CN}
Let $G=C_{n}(S)$ be a non-complete circulant graph. Then
$$\omega(G) \leq \Big\lfloor \dfrac{n}{2} \Big\rfloor .$$
\end{Lemma}
\begin{proof}
Suppose that $\omega(G) > \lfloor \frac{n}{2} \rfloor$. So there exists a clique $F$ of cardinality $\lfloor \frac{n}{2} \rfloor +1$. Let $r \in \lbrace 1,2,\ldots ,\lfloor \frac{n}{2} \rfloor \rbrace$. For all $v \in F$, the set $\{v+r \ : \ v\in F \}$ contains $\lfloor \frac{n}{2} \rfloor + 1$ vertices so at least one of them belongs to $F$. Hence there exist $v,w \in F$, such that $w=v+r$.
Since $F$ is a clique $\{v,w=v+r\}\in E(G)$, that is $r \in S$. The latter works for any $r$, then we conclude 
$$S =  \Big\lbrace 1,2, \ldots, \Big\lfloor \frac{n}{2} \Big\rfloor \Big\rbrace,$$ so $G$ is complete, and this contradicts our assumption.
\end{proof}
\noindent Thanks to the Lemma \ref{CN}, we prove the following

%\begin{Theorem}\label{Th}
%Let $G=C_{n}(S)$ be a non-empty circulant graph with $n=p^{k}, \ k \in \mathbb{N}, \ p \ prime$, then
%$\tilde{\chi}(\Delta(G)) \neq 0.$
%\end{Theorem}

\begin{Theorem}\label{Th}
Let $p$ be a prime and let $G$ be a non-empty circulant graph on $n=p^{k}$ vertices with $k>0$. Then
$\tilde{\chi}(\Delta(G)) \neq 0.$
\end{Theorem}
\begin{proof}
By Lemma \ref{Div} it follows
$$i \cdot f_{i-1} = p^{k} \cdot f_{i-1,0} \\ \mbox{ with } \\ 0\leq i\leq d.$$
Since the graph $G$ is not empty, its complement graph $\bar{G}$ is not complete. Hence by Lemma \ref{CN}, we have that a maximum clique in $\bar{G}$ has cardinality $d~<~\frac{p^{k}}{2} $, namely $f_{i-1}$ is a non-zero multiple of $p$ for $1\leq i\leq d$.
Therefore
$$\widetilde{\chi}(\Delta(G))=\sum\limits_{i=1}^{d} (-1)^{i-1} f_{i-1} -1=pr -1$$
with $r \in \mathbb{Z}$. By the primality of $p$, $\widetilde{\chi}(\Delta(G))$ is always non-zero.
\end{proof}

Before stating the theorem on the case $n=2p^k$, we prove some properties that will be helpful.
\begin{Lemma}\label{2q}
Let $n=2q$ for an odd $q > 1$ and let $G=C_{n}(S)$ be a non-complete circulant graph. Then
$\omega (G) < q  \mbox{ if and only if }  \lbrace 2,4,\ldots, q-1 \rbrace \nsubseteq S.$
\end{Lemma}
\begin{proof}
$(\Rightarrow)$. By contraposition assume $ \lbrace 2,4,\ldots, q-1 \rbrace \subseteq S$. We observe that the set $\{0,2,4,\ldots ,2q-2\}$ is a clique of cardinality $q$. It negates the hypotesis.

\noindent $(\Leftarrow)$. Let $r \in \lbrace 2,4,\ldots, q-1 \rbrace$ be such that $r \notin S$. Let $C$ be the set of vertices of a clique of $G$. We claim $|C|<q$.
We partition the set of vertices $V(G)=\lbrace 0,1,2,\ldots ,2q-1 \rbrace$
in the two sets 
\[
 V_1=\lbrace 2k \ | \ k=0, \ldots ,q-1 \rbrace \mbox{ and } V_2=\lbrace 2k+1 \ | \ k=0,\ldots ,q-1 \rbrace . 
\]
We observe that $|V_1|=|V_2|=q$.

\noindent
\textit{Claim:} $V_{1}$ (respectively $V_{2}$) contains at most $\frac{q-1}{2}$ vertices such that for each pair $v,w\in V_{1}$ we have $|v-w|_{n}\neq r$.

\noindent
\textit{Proof of the Claim.} By contraposition, assume that we take a subset $V'$ of cardinality $\frac{q-1}{2}+1$ of $V_{1}$ with the desired property, say 
\[
 V'=\{v_{1},v_{2},\ldots,v_{\frac{q-1}{2}+1} \},
\]
and since $V' \subset V_{1}$, these are all even vertices. Now we take the set $V''=\{ v+r: v \in V' \}$. Since $r$ is even, $V'' \subset V_{1}$ and $|V'|=|V''|=\frac{q-1}{2}+1$. Since $|V'|+|V''|>q$, then $V' \cap V'' \neq \varnothing$, so there exist $v,w \in V'$ such that $w=v+r$; hence, the set $V'$ has not the desired property.  The claim follows.

\noindent Thus $|C \cap V_1|\leq \frac{q-1}{2}$ and $|C \cap V_2|\leq \frac{q-1}{2}$, so that $|C| < q$.
\end{proof}
\noindent We give a generalization of Lemma \ref{2q} in the following
\begin{Lemma}\label{dq1}
Let $n=rq$ for an odd $q > 1$. Let $G=C_{n}(S)$ be a non-complete circulant graph. Then: 
\begin{enumerate}
 \item[1)] If  $\Big\lbrace r,2r, \ldots , \frac{q-1}{2} r \Big\rbrace \nsubseteq S$  then  $\omega (G) \leq \frac{n-r}{2}$. 
\item[2)] If  $\omega (G) < q$  then  $\Big\lbrace r,2r, \ldots , \frac{q-1}{2} r  \Big\rbrace \nsubseteq S$.
 \end{enumerate}
\end{Lemma}
\begin{proof}
$1)$ The proof follows the steps of ($\Leftarrow$) of Lemma \ref{2q}. We assume $jr \notin S$ for some $j$, $1 \leq j \leq \frac{q-1}{2}$.
In this case we consider the partitions $V_i$ of $V(G)$
\[
 V_i=\{rk +i \ | \ k=0,\ldots,q-1\}
\]
with $i=0,\ldots, r-1$.
Let $C$ be the set of vertices of a clique of $G$. By using similar arguments to the Claim inside the proof of Lemma \ref{2q}, we can choose at most $\frac{q-1}{2}$ vertices within each $V_i$ such that for each pair $v,w\in V_{i}$ we have $|v-w|_{n}\neq jr$. Hence for any $i\in\{0,\ldots, r-1\}$, it follows that $|C \cap V_i|\leq \frac{q-1}{2}$. Since $r \cdot( \frac{q-1}{2}) =\frac{n-r}{2}$, at the end we get $|C|\leq  \frac{n-r}{2}$.

\noindent $2)$ The same proof of Lemma \ref{2q} ($\Rightarrow$) holds.
\end{proof}

\begin{Remark}
We highlight that by plugging $r=2$ in $1)$ and $2)$ of Lemma \ref{dq1}, we obtain the two implications of Lemma \ref{2q}. It is the unique case of $n=rq$ such that $\frac{rq-r}{2}$, the bound in $1)$, is equal to $q-1$, the bound in $2)$.
\end{Remark}

For the sake of simplicity, in Proposition \ref{congruent2} and Example \ref{exa} we focus our attention on the clique complex of the graph.

\begin{Proposition}\label{congruent2}
%Let $n=2p^k$ with $p$ odd prime, $k > 0$, 
Let $n=2p^k$ for an odd prime $p$, with $k > 0$, and let $G=C_{n}(S)$ be a circulant graph. If $f_{p^k-1}$, the number of cliques of cardinality $p^{k}$, is non-zero  then
\[
 f_{p^k-1} \equiv 2 \ mod  \ p.
\]
In particular, if one of the following condition holds
\begin{itemize}
\item[a)] $1 \notin S$,
\item[b)] $1 \in S$ and there exists $t \in \{ 1, \ldots ,p^k \} $ with $\gcd(t,2p)=1$ such that $t \notin S$,
\end{itemize}
then $f_{p^k -1}=2$.
\end{Proposition}

\begin{proof}
First suppose that the graph is complete. Since $f_{p^k-1}=\binom{2p^k}{p^k}$ and by Lucas's Theorem \cite{Lu}, we obtain
$$f_{p^k-1}= \binom{2p^k}{p^k} \equiv 2 \ mod \ p.$$
Now suppose that $G$ is not complete. By Lemmas \ref{CN} and \ref{2q} since $f_{p^k-1}\neq 0$, that is $\omega(G)=p^k$, we have that $\{2,4,\ldots , p^k-1 \}\subseteq S$. So $f_{p^k-1} \geq 2$. In fact the graph contains at least the two maximal cliques
\[
V_{1}=\{ 0,2,4,6, \ldots ,2p^k-2\} \ \ \mbox{and } \ V_{2}=\{ 1,3,5,7, \ldots ,2p^k-1\}.
\]
We observe that each clique of cardinality $p^k$ different from $V_1$ and $V_2$ has non-empty intersection with $V_1$ and $V_2$. \\
We first study the particular cases for which $f_{p^k -1}=2$. Suppose $1\not\in S$.
In a clique $V$ of $p^k$ vertices different from $V_{1}$ and $V_{2}$, we must have $p^k$ intervals between two consecutive vertices in $V$ containing at least 1 vertex not in $V$, except for one containing 2 vertices, otherwise $V$ could be identified with $V_{1}$ or $V_{2}$. It follows $|V(G) \setminus V| \geq p^k-1+2=p^k+1$, that yields $|V|<p^k$. It contradicts the assumption. 

Now suppose that $1 \in S$ and there exists $t$ odd and coprime with $p$, $3 \leq t < p^k$, such that $t \notin S$. We prove $f_{p^k-1} =2.$
By contraposition let $V$ be a clique of cardinality $p^k$ different from $V_{1}$ and $V_{2}$, containing $0$ and $1$. Let $V'=\{v+t\:\ v\in V\}$. 
If $V \cap V' \neq \varnothing$ there exist $v,w \in V$ such that $w=v+t$. It is impossible. If $V \cap V'= \varnothing$, then $V(G)= V \sqcup V'$. Since $(t,p)=1$ and $t$ odd, then $(t,n)=1$, hence there exists an odd $a \in \mathbb{Z}_{n}$, coprime with $n$, such that $at \equiv 1 \ mod \ n $. Since $0 \in V$ and $t \in V'$ by definition of $V'$ we have that $2t \in V$. In fact if $t+t=2t \in V'$, then $t \in V$, obtaining a contradiction.
It follows that 
\[
 2bt \in V \mbox{ and } (2b+1)t \in V' \ \ \mbox{for any } b.
\] 
The vertex $at$ lives in $V'$ since $a$ is odd and lives in $V$ since $at=1\in V$. It implies $t \in S$. It is false.\\
Hence a clique of $p^k$ vertices different from $V_{1}$ and $V_{2}$ cannot exist and $f_{p^k -1}=2$. 

\noindent We assume $f_{p^k -1} >2$. Then by the previous observations we have 
\[
\Big\{t: t  \mbox{ odd and } (t,p)=1 \Big\} \cup \Big\{2k: k=1, \ldots , \frac{p^k-1}{2} \Big\} \subseteq S.
\]
Now we distinguish two cases:
\begin{enumerate}
 \item[1)] $S$ is the ($p^k-1$)-th power cycle, namely $S=\{1,\ldots , p^k -1 \};$
 \item[2)] $S$ is not the ($p^k-1$)-th power cycle.
\end{enumerate}
\noindent
$1)$ In this case $f_{p^k-1}$ is the coefficient of the degree $p^k$ term of the independence polynomial of the graph $C_{n}(p^k)$. As pointed out after Definition 3.4 in \cite{BH0}, this polynomial is 
$$(1+2x)^{p^k}$$
hence $f_{p^k-1}=2^{p^k} \equiv_{p} 2$ by Fermat's Little Theorem.

\noindent
$2)$ If $S\neq \{1,\ldots , p^k -1 \}$, there exists an odd multiple of $p$, $mp$ for an odd $m$ with $mp < p^k$, such that $mp \notin S$. Let $m=q \cdot p^{r}$ for some odd $q$ with $\gcd(q,p)=1$ and $0 \leq r < k-1$. Let $V$ be a clique of cardinality $p^k$ different from $V_1$ and $V_2$. Let $V'=V+mp:=\{v+mp \ : v \in V\}.$ We have $V \cap V'=\varnothing$. Moreover if $v \in V$ then $v+mp \in V'$ and $v+2mp\in V$. In fact if $v+2mp \in V'$ then $mp \in S$ since $V'$ is a clique. The latter implies that $V=V+2mp=V+2qp^{r+1}$. Since $q$ is odd and coprime with $p$, it is coprime with $n$, hence it is invertible in $\mathbb{Z}_{n}$. Therefore, there exists $h\in \mathbb{Z}_{n} $ such that $qh \equiv_{n} 1$, and $2qhp^{r+1}\equiv_{n} 2p^{r+1}$. Since $V=V+2qp^{r+1}$, then $V=V+2p^{r+1}$.
Now we prove that if $j \in \mathbb{Z}_{n}$ is such that 
\[
V= V+j,
\]
then $2p \ | \  j$.
By contradiction assume $V=V+j$ and $2p  \nmid  j$. 
We write $j=2p^{r+1}a+b$ with $0 < b < 2p^{r+1}$ and $2p \nmid \ b$.
Since $V=V+j$, then $V=V+b$ and we have
\[
L=\{0,b,2b,\ldots (o(b)-1)b\}\subseteq V.
\]
where $o(b)$ is the order of $b$ in $(\mathbb{Z}_n,+)$.
We analyze $g=\gcd(b,2p^{r+1})$ to determine the order of $b$ in $\mathbb{Z}_{n}$. Since $2p \nmid b$ and $b< 2p^{r+1}$, $g$ could be either $1$, $2$, $p^{i}$ with $1 \leq i \leq r+1$. \\
If $g=1$, then $o(b)=n$ and $L=\mathbb{Z}_{n}$, but $|V|=\frac{n}{2}$. It is impossible.\\
If $g=2$, then $o(b)=p^{k}$, $L=V_{1}\subseteq V$ and $|V_{1}|=|V|$ hence $V=V_{1}$. It is a contradiction to the assumption $V\neq V_{1}$.\\
If $g=p^{i}$, then $V=V+p^{i}=V+qp^{r+1-i}p^{i}=V'$. It contradicts the fact $V \cap V'=\varnothing$.\\
Hence, if $2p \nmid  j $ then $V\neq V+j$. Let $s$ be the minimum positive integer such that $V=V+2sp$.
Since $V=V+2p^{r+1}$, it follows that  $s \leq p^r$, $2p \leq 2sp < 2p^k$, and
\[V,V+1,\ldots ,V+(2sp-1)\]
are $2sp$ different cliques of $G$ having cardinality $p^k$. Hence $2p$ divides $(f_{p^k-1}-2)$
and $$f_{p^k-1} \equiv_{p} 2.$$
The assertion follows.
%This proves that for any $S$, $f_{p^k -1} \equiv 2 \ \mbox{mod } p.$
\end{proof} 

\begin{Example}\label{exa}
We provide an example with $f_{p^k-1}>2$ and $S \neq  \{1,2,\ldots,p^k~-~1\}.$ We consider the graph $G=C_{50}( \{1,2,\ldots ,24\}\setminus \{5\})$, using the notation of the proof of Proposition \ref{congruent2}, $V=V+2mp=V+2\cdot 1\cdot 5=V+10$. The clique complex of $G$ has $32$ cliques of cardinality $25$. 
We fix a vertex $v$, for simplicity $0$, and we look at the sequence of vertices $0,1,\ldots, 9$. Moreover, the symbol $0$ denotes a vertex not in a clique, while the symbol $1$ refers to a vertex in a clique. We have that $V_{1}$ has fundamental pattern $1010101010$, while $V_{2}$ has fundamental pattern $0101010101$. Since $V=V+10$, each fundamental pattern is repeated $5$ times to cover all the vertices of the graph. For example,
\[
V_{1}=1010101010.1010101010.1010101010.1010101010.1010101010
\]
and it happens for all the other cliques.
The complex has $30$ further cliques of three kinds, namely there are three further different patterns in the sequence of vertices in or not in a clique. The other three fundamental patterns are $1111100000, \ 1110100010$ and $1101100100$.
Since the graph is circulant, for each of the last three patterns, there are $10$ different cliques corresponding to the $10$ possible shifts. For example, for the first sequence we will have \[
0111110000, \ 0011111000, \ldots , 1111000001. \] 
So the total number of cliques will be $3 \cdot 10 +2=32$. 
\end{Example}

\noindent Now we are able to prove 
%\begin{Theorem}\label{Th2}
%Let $n=2p^k$ with $p$ odd prime, $k > 0$, and $C_{n}(S)$ be a non-empty circulant graph. Then $\widetilde{\chi}(\Delta (G)) \neq 0$.
%\end{Theorem}
\begin{Theorem}\label{Th2}
Let $p$ be an odd prime and let $G$ be a non-empty circulant graph on $n=2p^{k}$ vertices with $k>0$. Then $\widetilde{\chi}(\Delta (G)) \neq 0$.
\end{Theorem}
\begin{proof}
By using similar arguments to Theorem \ref{Th} we say that 
\[
p  \ | \ f_{i-1} \ \ \mbox{with }   1 \leq i \leq p^k -1 
\]
So we write
$$\widetilde{\chi}(\Delta (G))= pt+f_{p^k-1}-1 \ 	\ \ \mbox{for some } t \mbox{ in } \mathbb{Z}.$$
Since by Proposition \ref{congruent2} $f_{p^k-1}$ is $0$ or it is congruent $2$ modulo $p$, we have
$$\widetilde{\chi}(\Delta(G))=pr \pm \ 1$$
for some $r$ in $\mathbb{Z}$.
That is $\widetilde{\chi}(\Delta (G))$ does not vanish.
\end{proof} 

\begin{Example}\label{counterexample}
In the proofs of Theorems \ref{Th} and \ref{Th2} we are giving a partial positive answer to the Conjecture 5.38 of \cite{Ho} stating that all non-empty circulant graphs $G$ have $\widetilde{\chi}(\Delta(G))\neq 0$. 
But with a \texttt{\emph{MAGMA}} algorithm, available at 
\begin{center}
\texttt{http://www.giancarlorinaldo.it/eulercirculants.html},
\end{center}
we have found for $n=30$ and $n=36$ a list of circulants, up to isomorphisms, that contradict the Conjecture (see Table \ref{tab}).
Among those, for example, we report the circulant $G=C_{30}(\{1,3,8\})$ whose independence complex has $f$-vector equals to
$ [ 1, 30, 345, 1990, 6360, 11736, 12600, 7680, 2430, 300 ]$.
That is $$\widetilde{\chi}(\Delta (G))=-1+30-345+1990-6360+11736-12600+7680-2430+300=0.$$
\end{Example}
\begin{table}[h]
\centering
\begin{small}
\begin{tabular}{| ccc |}
\hline
  &$n=30$  &  \\
\hline
$\{ 1, 3, 8 \}$ & $\{ 1, 7, 9, 11, 14 \}$& $\{ 1, 2, 3, 7, 9, 11, 13 \}$\\
$\{ 2, 9, 13 \}$  &  $\{ 1, 4, 9, 13, 14 \}$ &$\{ 2, 3, 4, 5, 7, 9, 14 \}$\\
  $\{ 8, 9, 13 \}$ & $\{ 2, 3, 7, 8, 9 \}$ &$\{ 2, 3, 4, 5, 8, 9, 14 \}$\\
  $\{ 1, 8, 9, 14 \}$ & $\{ 1, 3, 4, 9, 11  \}$& $\{ 1, 3, 4, 5, 7, 8, 14 \}$\\
    $\{ 2, 3, 11, 13 \}$ & $\{ 2, 7, 8, 9, 13  \}$ &$\{ 2, 3, 4, 5, 8, 11, 13 \}$\\
    $\{ 3, 8, 11, 13 \}$ & $\{ 2, 3, 4, 7, 13  \}$ &$\{ 1, 2, 3, 7, 8, 9, 11, 13 \}$\\
    $\{ 1, 3, 4, 13 \}$ & $\{ 1, 3, 4, 5, 7, 8  \}$ &$\{ 2, 3, 5, 8, 9, 11, 13, 14 \}$\\
    $\{ 7, 8, 9, 13 \}$ &$\{ 2, 3, 4, 5, 8, 11  \}$&$\{ 1, 2, 3, 4, 5, 8, 9, 14 \}$\\
    $\{ 1, 4, 7, 9 \}$ &$\{ 1, 2, 3, 8, 9, 11  \}$ &$\{ 1, 2, 3, 5, 7, 9, 11, 14 \}$\\
    $\{ 1, 8, 9, 11 \}$ &$\{ 1, 3, 4, 7, 9, 13  \}$ &$\{ 2, 3, 4, 5, 7, 9, 13, 14 \}$\\
    $\{ 2, 9, 11, 14 \}$ &$\{ 1, 4, 7, 9, 11, 14  \}$ &$\{ 1, 3, 4, 5, 7, 8, 9, 11, 13 \}$\\
    $\{ 1, 2, 9, 13 \}$ &$\{ 1, 2, 3, 5, 11, 14  \}$ &$\{ 1, 2, 3, 4, 5, 8, 9, 11, 13 \}$\\
    $\{ 2, 3, 7, 9 \}$&$\{ 1, 3, 4, 9, 11, 14  \}$ &$\{ 2, 3, 4, 5, 7, 8, 9, 13, 14 \}$\\
    $\{ 1, 7, 8, 9, 11 \}$ &$\{ 2, 3, 4, 7, 8, 13  \}$ &$\{ 1, 2, 3, 4, 5, 7, 9, 11, 13, 14 \}$\\
    $\{ 1, 3, 7, 8, 13 \}$ &$\{ 1, 2, 5, 7, 9, 13, 14  \}$  & \\
    $\{ 2, 3, 4, 7, 8 \}$ &$ \{ 1, 4, 5, 7, 8, 9, 11  \}$  & \\
    \hline 
    &$n=36$ & \\
    \hline
$\{ 2, 3, 6, 7, 10, 14, 15 \}$ &$\{ 2, 5, 6, 7, 10, 11, 14 \}$  &$\{ 2, 5, 6, 10, 11, 13, 14 \}$ \\
$\{ 1, 2, 5, 6, 7, 10, 11, 17 \}$ &$\{ 1, 5, 6, 7, 11, 13, 14, 17\}$ &$\{ 2, 5, 6, 7, 10, 14, 15, 17 \}$ \\
 $\{ 1, 2, 5, 6, 7, 10, 11, 13 \}$ &$\{ 1, 5, 6, 7, 10, 11, 13, 14, 17 \}$ & \\
\hline
\end{tabular}
\end{small}
\caption{\fontsize{9}{15}\selectfont The table shows $G=C_{n}(S)$ such that $\widetilde{\chi}(\Delta(G))=0$, up to isomorphisms.}\label{tab}
\end{table}

We now give some applications.
\noindent The structure and roots of the independence polynomial have been studied by Hoshino and Brown (see \cite{BHN}, \cite{Ho}, \cite{BH0}). By Theorem \ref{Th} and Theorem \ref{Th2}, we obtain the following
\begin{Corollary}
Let $n\in \{p^k, 2p^k\}$ for a prime $p$ and for $k > 0$, and let $G$ be a non-empty circulant graph on $n$ vertices. Then
\[
 I(G,-1)\neq 0.
\]
\end{Corollary}
\noindent By Example \ref{counterexample} and Equation \eqref{charind}, $-1$ is a root of the independence polynomial of the circulant graph $C_{30}(1,3,8)$.

\noindent Similar results follow by Remark \ref{remri} for the regularity index and the $a$-invariant. Moreover by using Corollary 4.8 of \cite{Ei}, we get the following
\begin{Corollary}
Let $G$ be a circulant graph as in Theorem \ref{Th} and Theorem \ref{Th2}. If $G$ is Cohen-Macaulay then
\[
\depth R/I(G)= \reg R/I(G).
\]
\end{Corollary}

%----STRONG CONCLUSION!!-----
It is of interest to find the hypothesis, also of different type, for the non-vanishing of the reduced Euler characteristic of circulant graphs.

We focused on the property of $n$, but nice combinatorial properties like well-coveredness (see \cite{Ho}, \cite{EMT}),  strongly connectedness (see \cite{Ri}), vertex decomposability and shellability (see \cite{MTW}) could be helpful.
From another point of view, it would be nice to find entire classes of circulants that for particular $n$ and $S$ have vanishing Euler characteristic, by using a theoretical approach rather than the computational one used in Example \ref{counterexample}.

\end{document}